\documentclass[11pt]{amsart}

\usepackage{amssymb,amsmath,amsthm} 
\usepackage{stmaryrd} 
\usepackage{mathrsfs} 

\usepackage{enumitem} 
\setenumerate{label=(\alph*), font=\normalfont}

\usepackage[all]{xy} 

\usepackage{multirow} 
\usepackage{tabu}
\usepackage{xcolor}
\usepackage{url}
\usepackage{hyperref}

\theoremstyle{plain}
\newtheorem{thm}{Theorem}[section]
\newtheorem{lem}[thm]{Lemma}
\newtheorem{prop}[thm]{Proposition}
\newtheorem{cor}[thm]{Corollary}
\newtheorem{conj}[thm]{Conjecture}

\theoremstyle{definition}
\newtheorem{defn}[thm]{Definition}

\theoremstyle{remark}
\newtheorem{rem}[thm]{Remark}

\numberwithin{equation}{section}

\newcommand{\bbk}{\Bbbk}
\newcommand{\bbQ}{\mathbb{Q}}
\newcommand{\bbP}{\mathbb{P}}

\newcommand{\bbC}{\mathbb{C}}
\newcommand{\bbG}{\mathbb{G}}
\newcommand{\bbZ}{\mathbb{Z}}
\newcommand{\bbF}{\mathbb{F}}

\newcommand{\Cr}{\operatorname{Cr}}
\newcommand{\Spec}{\operatorname{Spec}}
\newcommand{\GL}{\operatorname{GL}}
\newcommand{\SL}{\operatorname{SL}}
\newcommand{\PGL}{\operatorname{PGL}}
\newcommand{\PSL}{\operatorname{PSL}}

\newcommand{\rdim}{\operatorname{rdim}}
\newcommand{\fdim}{\operatorname{fdim}}
\newcommand{\Aut}{\operatorname{Aut}}
\newcommand{\sheafAut}{\underline{\operatorname{Aut}}}

\newcommand{\cha}{\operatorname{char}}
\newcommand{\id}{\operatorname{id}}
\newcommand{\Hom}{\operatorname{Hom}}
\newcommand{\Pic}{\operatorname{Pic}}

\begin{document}

\title{Representations of finite subgroups of Cremona groups}

\author[A.~Duncan]{Alexander Duncan}
\address{Department of Mathematics, University of South Carolina, 
Columbia, SC 29208, USA}
\email{duncan@math.sc.edu}

\author[B.~Heath]{Bailey Heath}
\address{Department of Mathematics, Yale University,
New Haven, CT 06511, USA}
\email{bailey.heath@yale.edu}

\author[C.~Urech]{Christian Urech}
\address{Department of Mathematics, ETH Zurich, 8092 Zurich, Switzerland}
\email{christian.urech@math.ethz.ch}

\date{\today}

\begin{abstract}
The Cremona group of rank n over a field k is the group of
birational automorphisms of the n-dimensional projective space
over the field k.
We study the minimal dimension such that all finite subgroups of
the Cremona group have a faithful representation of that dimension
over the same field.
We find the exact value for rank 1 and 2 over all fields.
We prove that the value is infinite for all fields of positive
characteristic and rank greater than one.
For many fields of characteristic 0, which include number fields
and the complex field, we show that the value is finite for all ranks.
Finally, for all fields of characteristic 0,
we prove that the dimension is bounded below by a function
that is exponential in the rank.
\end{abstract}

\maketitle

\section{Introduction}
\label{sec:intro}

Let $\bbk$ be a field.
For a positive integer $n$,
the \emph{Cremona group of degree $n$ over $\bbk$},
denoted $\Cr_n(\bbk)$, is the group of birational automorphisms of the
projective space $\bbP^n_\bbk$.
Equivalently, $\Cr_n(\bbk)$ is the group of $\bbk$-algebra automorphisms of
the purely transcendental field extension $\bbk(x_1, \ldots, x_n)/\bbk$.

The complex plane Cremona group $\Cr_2(\bbC)$
has been studied for more than 150 years.
Classically, finite subgroups of Cremona groups have been of particular interest.
The finite subgroups of $\Cr_2(\bbC)$ are (almost) completely
classified \cite{Blanc,DolIsk}.
Even in this case, the classification is quite complicated. 
Recently, there has been work towards classifying finite subgroups
over fields $\bbk$ of positive characteristic
\cite{cubicsPosChar,DolgachevMartin24,DolgachevMartin25},
for non-closed fields
\cite{Yasinsky16,Yasinsky22,Boitrel23,Zaitsev23,Smith23,Smith25} 
and for higher $n$
\cite{Pro12Simple,Prok2Group,CremonaSpacePGroup,BCDP}.

Rather than a full classification, in many applications all that
suffices is a bound on the complexity of the finite subgroups that
occur.
Minkowski \cite{Minkowski} found the least common
multiple of the orders of the finite subgroups of
the general linear group $\GL_n(\bbQ)$.
Minkowski's bound can be extended to other
algebraic groups over many other fields \cite{SerreMinkowskiG}.
In \cite{SerreMinkowskiCr}, Serre found a Minkowski-style bound for
$\Cr_2(\bbk)$ for many fields $\bbk$, such as finite fields and number fields.
Unpublished work of Feit and Weisfeiler found absolute upper bounds
on the orders of finite subgroups of every $\GL_n(\bbQ)$
(see \cite{Weisfeiler} and \S{6.1}~of~\cite{GuralnickLorenz}).
However, when $\bbk=\bbC$, even the finite subgroups of
$\GL_1(\bbC) \cong \bbC^\times$ do not have bounded order.

Jordan proved that for every positive integer $n$ there is a positive
integer $J(n)$ such that, if $G$ is a finite subgroup of $\GL_n(\bbC)$,
then $G$ has a normal abelian subgroup of index $\le J(n)$.
The number $J(n)$ is called the \emph{Jordan constant} of $\GL_n(\bbC)$.
Using the classification of finite simple groups,
Collins \cite{Collins} explicitly computed $J(n)$ for all $n$.
One can ask about existence and values of the Jordan constant
for Cremona groups as well.
By work of Birkar~\cite{Birkar} and Prokhorov and Shramov~\cite{PSJordan},
every Cremona group in characteristic $0$ has a finite Jordan constant.
However, explicit bounds on Jordan constants are only known for
$n \le 3$ and a few fields \cite{PSJordan,Yasinsky17}.

The work above use the Minkowski and Jordan bounds from the general
linear group as inspiration for similar bounds for Cremona groups.
In this paper, we bound the complexity of finite subgroups of
Cremona groups via general linear groups directly.   
The \emph{representation dimension} of a finite group $G$ over $\bbk$, denoted
$\rdim_\bbk(G)$, is the minimal $N$ such that there is an embedding
$G \hookrightarrow \GL_N(\bbk)$.
For a field $\bbk$ and a positive integer $n$, define
\[
 c_n(\bbk) :=
\sup \left\{ \rdim_\bbk(G) \ | \ G \textrm{ finite group such that }
G \subseteq \Cr_n(\bbk) \right\} \ .
\]
This paper begins a systematic study of $c_n(\bbk)$ for various integers
$n$ and fields $\bbk$.

The line Cremona group $\Cr_1(\bbk)$ is isomorphic to the
linear algebraic group $\PGL_2(\bbk)$.  Even in this case,
the representation theory of its finite subgroups may differ from that
of the algebraic group itself:

\begin{thm} \label{thm:c1}
\[
c_1(\bbk) = 
\begin{cases}
2 & \textrm{if $\cha(\bbk)=2$,}\\
3 & \textrm{if $\cha(\bbk)\ge 3$,}\\
3 & \textrm{if $\cha(\bbk) = 0$ and $-1$ is a sum of two squares,}\\
2 & \textrm{otherwise}.
\end{cases}
\]
\end{thm}

For larger $n$, the group $\Cr_n(\bbk)$ is much more complex.
Here the group $\Cr_n(\bbk)$ is not an algebraic group.
Indeed, if $\bbk$ is infinite, then $\Cr_n(\bbk)$
does not have an embedding into $\GL_n(\bbF)$
for any $n\geq 1$ and any field $\bbF$,
even abstractly (see \cite[5.1]{CerveauDeserti}, \cite{Cornulier}).

However, the finite subgroups may be much more manageable.
In the complex case, the third author showed that $c_2(\bbC) \le 48$
(see\cite[Theorem~1.7]{Urech21} and \cite{urech2017subgroups}).
We find explicit values for the plane Cremona groups over
all fields:

\begin{thm} \label{thm:c2}
\[
c_2(\bbk) = 
\begin{cases}
\infty & \textrm{if $\cha(\bbk) \ne 0$,}\\
8 & \textrm{if $\cha(\bbk) = 0$ and $\sqrt{-3} \in \bbk$,}\\
6 & \textrm{otherwise}.
\end{cases}
\]
\end{thm}

For the complex space Cremona group we have bounds
\[
15 \le c_3(\bbC) \le 62208
\]
from Theorem~\ref{thm:c3_bounded} below.

For general $n$, we do not have explicit upper bounds.
Indeed, we show that the representation dimension is not bounded in positive
characteristic.

\begin{thm} \label{thm:infiniteness}
If $n \ge 2$, then
$c_n(\bbk)$ is infinite if $\cha(k)>0$.
\end{thm}

In contrast, as a corollary of the work of Birkar, Prokhorov and
Shramov on the Jordan constant,
we have finiteness in characteristic $0$ for many fields:

\begin{thm} \label{thm:finiteness}
Suppose $\bbk$ is a field of characteristic $0$ that either
contains all roots of unity or is finitely generated over $\bbQ$.
Then $c_n(\bbk)$ is finite.
\end{thm}

Explicit upper bounds for $c_n(\bbk)$ for all $n$ appear out of reach
at this time as they would imply explicit upper bounds for the
corresponding Jordan
constants.
However, we have the following uniform lower bounds obtained
using the representation theory of algebraic tori:

\begin{thm} \label{thm:lowerBounds}
The following lower bounds hold for all fields $\bbk$:\newline
\begin{center}
\begin{tabular}{c|ccccccc}
$n$ & $1$ & $2$ & $3$ & $4$ & $5$ & $6$ & $\ge 7$\\
\hline
$c_n(\bbk) \ge$ & $2$ & $6$ & $12$ & $24$ & $40$ & $72$ & $2^n$\\
\end{tabular}
\end{center}
\end{thm}

The structure of the paper is as follows.
After preliminaries in \S{\ref{sec:prelim}},
we prove Theorem~\ref{thm:c1} on the Cremona group of degree $1$
in \S{\ref{sec:cr1}}.
We prove Theorem~\ref{thm:infiniteness}
in \S{\ref{sec:poschar}}; the remainder of the paper is in
characteristic $0$.
In \S{\ref{sec:torus}}, we prove the lower bounds of
Theorem~\ref{thm:lowerBounds}.
In \S{\ref{sec:pgl}} and \S{\ref{sec:wps}}, we establish some basic
results on projective spaces and weighted projective spaces.
These results are applied in \S{\ref{sec:cr2}} where we
finally prove Theorem~\ref{thm:c2} on the Cremona group of degree $2$.
In \S{\ref{sec:cr_higher}}, we prove
Theorem~\ref{thm:finiteness} and make some remarks about higher
degree Cremona groups.
In Appendix~\ref{sec:tedious_group_theory}, we prove a technical group
theoretic result used in \S{\ref{sec:cr2}}.

\subsection*{Acknowledgements}

The authors thank R.~J.~Shank for helpful ideas used in
Section~\ref{sec:poschar}.
A.~Duncan and C.~Urech would like to thank the Mathematisches Forschungsinstitut
Oberwolfach and the Simons Center for Geometry and Physics
where some of the work for this paper was completed.

\section{Preliminaries}
\label{sec:prelim}

Throughout all of this article, a variety over a field $\bbk$ is a
geometrically integral and separated scheme of finite
type over $\bbk$. Surfaces are always assumed to be projective. 

Let $\bbG_m = \Spec(\bbZ[t,t^{-1}])$ be the multiplicative group scheme
and $\mu_n = \Spec(\bbZ[t]/(t^n-1))$ be the group scheme of $n$-th roots of
unity.
For a field $\bbk$, the $\bbk$-group scheme $\bbG_{m,\bbk}$ is an algebraic group
(i.e. smooth), while $\mu_{n,\bbk}$ is smooth if and only if $n$ is not
divisible by $\cha(\bbk)$.

By $\GL_n$ we denote the group scheme of invertible matrices,
and by $\SL_n$ the group scheme of invertible matrices with determinant
$1$. We define $\PGL_n$ to be the quotient scheme $\GL_n/\bbG_m$ and
$\PSL_n(k)$ the image of the map $\SL_n(k) \to \PGL_n(k)$.

Let us extend the notion of {representation dimension} to algebraic
groups.
For an algebraic group $G$ defined over $\bbk$, the
\emph{representation dimension} of $G$, denoted $\rdim_\bbk(G)$, is the
minimal dimension of an (algebraic) faithful representation of $G$ over
$\bbk$.
If $G$ is a constant finite group, then $\rdim_\bbk(G)=\rdim_\bbk(G(\bbk))$.

\begin{rem}
Note that $\GL_2(\bbF_2) \cong \PGL_2(\bbF_2)$ as abstract groups;
while $\GL_{2,\bbF_2} \not\cong \PGL_{2,\bbF_2}$ as algebraic groups.
This distinction is important since
it is possible that $\rdim_\bbk(G(\bbk)) \ne \rdim_\bbk(G)$ for certain
algebraic groups $G$.
For example, we will see below that
$\rdim_{\bbF_2}(\PGL_2(\bbF_2))=2$ as an abstract group,
while
$\rdim_{\bbF_2}(\PGL_{2,\bbF_2}))=3$
as an algebraic group. However, we always have $\rdim_\bbk(G(\bbk))\leq\rdim_\bbk(G)$.
\end{rem}

For an abstract group $G$, we define $\fdim_\bbk(G)$ to be the maximum of
$\rdim_\bbk(H)$ where $H$ is a finite subgroup of $G$.
For an algebraic group $G$, we define $\fdim_\bbk(G)=\fdim_\bbk(G(\bbk))$.

Suppose $X$ is a variety defined over a field $\bbk$.
Recall that a \emph{$\bbk$-form of $X$} is a $\bbk$-variety $Y$ such that
$X_K \cong Y_K$ for a field extension $K/\bbk$.

\begin{lem} \label{lem:forms}
Let $X$ be a $\bbk$-variety and suppose the automorphism group scheme
$\sheafAut(X)$ is a linear algebraic group.
Then 
\[
\rdim_k(\sheafAut(X)) = \rdim_k(\sheafAut(Y))
\]
for any $\bbk$-form $Y$ of $X$.
\end{lem}

\begin{proof}
We use the formalism of Galois cohomology and twisting by torsors
(see, for example, \S{1.5}~and~\S{3.1}~of~\cite{SerreGC}).

Since $Y$ is a $k$-form of $X$,
there exists an $\sheafAut(X)$-torsor $T$ such that $Y \cong {}^T\!X$
where ${}^T\!X := (T \times X)/G$ is the twist of $X$ by $T$.
Suppose there exists an embedding $f : \sheafAut(X) \to \GL_{n,\bbk}$
of algebraic groups.
There is a natural action of $\sheafAut(X)$ on itself by conjugation and,
via $f$, on $\GL_{n,\bbk}$ as well, making $f$ equivariant.
Twisting is functorial, so there is an embedding
${}^T\!f : {}^T\!\sheafAut(X) \to {}^T\!\GL_{n,\bbk}$
of algebraic groups.
Now ${}^T\!\sheafAut(X)$ is isomorphic to $\sheafAut({}^T\!X)$ and,
by Hilbert's Theorem 90, ${}^T\!\GL_{n,\bbk}$
is isomorphic to $\GL_{n,\bbk}$.
Thus we obtain an embedding $\sheafAut(Y) \to \GL_{n,\bbk}$.

Conversely, since $\sheafAut(Y) \cong {}^T\!\sheafAut(X)$, it is also a linear
algebraic group.
Thus the same argument as above, shows that the
existence of an embedding $\sheafAut(Y) \to \GL_{n,\bbk}$
implies the existence of an embedding $\sheafAut(X) \to \GL_{n,\bbk}$.
Thus the desired result holds.
\end{proof}

We will use the notation $\sheafAut(X)$ to denote the automorphism group
scheme of a $\bbk$-variety $X$ and the notation $\operatorname{Aut}(X)$
to denote the abstract automorphism group,
which is simply $\sheafAut(X)(\bbk)$.
The distinction is occasionally important below.

\section{Line Cremona Groups}
\label{sec:cr1}

Observe that $\Cr_1(\bbk) \cong \PGL_{2}(\bbk)$.
This follows immediately from the standard facts that any rational map
of smooth proper curves is regular and that the automorphism group of
$\bbP^1$ is $\PGL_2(\bbk)$.

\begin{proof}[Proof of Theorem~\ref{thm:c1}]
Observe that $\PGL_{2,\bbk}$ has a $3$-dimensional faithful
representation as an algebraic group via the adjoint representation.
Thus $c_1(\bbk) \le 3$ for all fields $\bbk$.

First, we consider the case where $\bbk$ has characteristic $\ne 2,3$.
Assume that $-1$ is the sum of two squares in $\bbk$.
Recall that every element of a finite field is a sum of two squares,
so this hypothesis always holds in positive characteristic.
From Proposition~1.1~of~\cite{Beauville},
we see that $S_4$ is a subgroup of $\PGL_2(\bbk)$.
Since $S_4$ has order coprime to the characteristic, standard results on
the representation theory of symmetric groups imply that $\rdim_k(S_4)=3$.

Now assume that $-1$ is \emph{not} the sum of two squares in $\bbk$, so
in particular $\cha{\bbk}=0$.
The symmetric group on $3$ letters has a faithful $2$-dimensional
representation via
\begin{equation}
\left\langle 
\begin{pmatrix}
0 & -1 \\ 1 & -1
\end{pmatrix},\
\begin{pmatrix}
0 & 1 \\ 1 & 0
\end{pmatrix}
\right\rangle \ .
\end{equation}
There is no $1$-dimensional representation of a non-abelian group,
thus $c_1(\bbk)\ge 2$.

If $\bbk$ has characteristic $3$, then $c_1(\bbk)=3$
since $\PGL_2(\bbF_3) \cong S_4$
has no faithful $2$-dimensional representations.
It remains to consider the case when $\bbk$ has characteristic $2$.

If $K$ has characteristic $2$ and every element in $K$ is a square,
then the natural morphism $\SL_2(K) \to \PGL_2(K)$ is an isomorphism.
Thus there exists a purely inseparable field extension $K/\bbk$
along with an embedding $\rho : G \to \SL_2(K)$ such that
the map $\SL_2(K) \to \PGL_2(K)$ takes $G$ into its image in the subgroup
$\PGL_2(\bbk)$.
If $\bbk$ is finite then there are no non-trivial purely inseparable
field extensions, so we may assume $\bbk$ is infinite.

Suppose that $\rho$ is not absolutely irreducible.
In other words, there exists a field extension $F/K$ such that
$\rho_F$ is decomposable.
This means that there is a choice of basis such that
\begin{equation} \label{eq:upper_triangular}
\rho_F(g) = \begin{pmatrix} a & b \\ 0 & a^{-1} \end{pmatrix}
\end{equation}
for every $g \in G$ where $a \in F^\times$ and $b \in F$.
Since $G$ is finite, every $a \in F^\times$ has odd order
and is contained in the relative closure $L$ of the prime field
$\mathbb{F}_2$ in $F$.  Observe that $L$ is a finite field.

If $b=0$ for every $g \in G$, then $G$ is a cyclic group of odd order
and $\rho_F$ is defined over $L$.  Since $\rho$ is defined over $L \cap
K$, and finite fields are perfect, $K=\bbk$ and $\rho$ is actually
defined over $\bbk$.
Thus, $G$ has representation dimension $2$ over $\bbk$.

Now, suppose there exists an element $g \in G$ for which $b \ne 0$.
Then $G$ contains a normal elementary $2$-subgroup $N$.
Since $K$ has characteristic $2$, every representation of $G$ has a
basis in which $N$ is strictly upper triangular;
in particular, this is true of $\rho$.
The only matrices that leave invariant $N$ must be of the form
\eqref{eq:upper_triangular} above.
This means that $\rho$ is already of the form
\eqref{eq:upper_triangular},
We can therefore assume $F=K$ and $L \subseteq \bbk$.

Thus, $G$ embeds into $K \rtimes L^\times$.
Since $K$ is an $L$-vector space and $G$ is finite,
we in fact have that $G$ is a subgroup of $L^n \rtimes L^\times$
where $L^\times$ acts on each summand $L$ by multiplication.
Since $\bbk$ is also an infinite-dimensional $L$-vector space, we see that
$L^n \rtimes L^\times$ is also a subgroup of $\bbk \rtimes \bbk^\times$.
Thus $G$ has an embedding into $\GL_2(\bbk)$.
(Note that the original representation $\rho$ may not descend to $\bbk$
in this case --- we may have to construct a new representation.)

Finally, we consider the case where $\rho$ is absolutely irreducible.
Let $\sigma$ be a representation of $G$ over $\bbk$
such that $\sigma_K$ contains $\rho$ as a subrepresentation.
We may assume $\sigma$ itself is simple.
In this case, the endomorphism algebra $A$ of $\sigma$ is a
separable field extension of $\bbk$.
Since $K$ splits $A$ and $K$ is inseparable,
we conclude that $A$ is already split.
Thus $\sigma$ is absolutely irreducible and we have $\sigma_K=\rho$.
Thus, the representation $\rho$ can be defined over $\bbk$.
\end{proof}

\section{Positive Characteristic}
\label{sec:poschar}

In this section, we prove Theorem~\ref{thm:infiniteness}.
This will be a consequence of the following:

\begin{thm}\label{posdim}
Let $\bbk$ be a field of positive characteristic. We define the group 
\[
G:= \bbk[x]\rtimes \bbk,
\] 
where the action of $\bbk$ on $\bbk[x]$ is given by $a\cdot f(x):= f(x+a)$.
Then there exists no $n$ such that every finite subgroup of $G$ can be
embedded into $\GL_n(\bbk)$.
\end{thm}

Note that the same group $G$ defined over a field of characteristic $0$ has no non-trivial finite subgroup
since every non-trivial element has infinite order.
Thus, the hypothesis on the characteristic  is crucial.

To see how this implies Theorem~\ref{thm:infiniteness},
observe that $G$ can be embedded in group of automorphisms of the field
$\bbk(x,y)$ by
\begin{align*}
x &\mapsto x + a\\
y &\mapsto y + f(x)
\end{align*}
where $a \in \bbk$ and $f \in \bbk[x]$.

Since $G$ is a subgroup of $\Cr_2(\bbk)$,
Theorem~\ref{posdim} immediately implies that
$c_2(\bbk) = \infty$ for $\bbk$ of positive characteristic.
Since there are embeddings $\Cr_2(\bbk) \subseteq \Cr_n(\bbk)$
for all $n \ge 3$, we conclude that $c_n(\bbk) = \infty$ as well.
Thus, it remains to prove Theorem~\ref{posdim}.

A natural approach to proving the theorem is to explicitly determine the
representation dimensions of finite subgroups of $G$.
This can be carried out using Theorem~2.3 of \cite{ChenShankWehlau}
which is proved using techniques from modular invariant theory.
Their approach has the major advantage of finding representation dimensions
explicitly.
The first and third authors of the present paper independently found the
following alternative proof using model theory, which we think is interesting
enough to include here.

First, we recall a fundamental result in model theory due to Mal'cev \cite{Malcev}.

\begin{defn}
Let $\{x_i\}_{i\in I}$ be a set of variables. A {\it condition} is an expression of the form $F(x_{i_1},\dots, x_{i_k})=0$ or an expression of the form
\[
F_1(x_{i_1},\dots, x_{i_k})\neq 0\vee F_2(x_{i_k},\dots,
x_{i_k})\neq0\vee\dots\vee F_l(x_{i_1},\dots, x_{i_k})\neq 0\ ,
\]
where $F$ and the $F_i$ are polynomials with integer coefficients.
\end{defn}

\begin{defn}
	A set of conditions $S$ is {\it compatible} if there exists a field $\bbk$ which contains values $\{z_i\}_{i\in I}$ that satisfy $S$.
\end{defn}

\begin{thm}[\cite{Malcev}]\label{compact}
	A set of conditions $S$ is compatible if and only if every finite subset of $S$ is compatible.
\end{thm}

Theorem \ref{compact} is an easy but central result from model theory,  the so called compactness theorem. In its more general form it states that a set of first order sentences has a model if and only if each of its finite subsets has a model. It follows for example from the completeness of first order theories. 
We will use it in exactly the same way as Mal'cev originally used it:

\begin{lem}\label{embeddlem}
Let $G$ be a group and $n>0$ such that every finitely generated subgroup
of $G$ can be embedded into $\GL_n(\bbk)$ for some field $\bbk$ of
characteristic $p$.
Then there exists a field $K$ of characteristic $p$ such that $G$ can be
embedded into $\GL_n(K)$.
\end{lem}

\begin{proof}
	To every element $f\in G$ we associate a $n\times n$ matrix of variables $(x^f_{ij})$. Consider the following set $S$ of conditions:
	\begin{enumerate}
		\item   $(x^f_{ij})(x^g_{ij})=(x^h_{ij})$ for all $f,g,h\in G$ such that $fg=h$;
		\item $(\bigvee_{i} x_{ii}^g-1\neq 0)\vee (\bigvee_{i\neq j}x^g_{ij}\neq 0);$ for all $g\in G\setminus\{\id\}$. 
		\item $x_{ii}^{\id}=1$ and $x_{ij}^{\id}=0$ for all $1\leq i,j\leq n$, $i\neq j$;
		\item $p= 0$.
	\end{enumerate}
First we note that if $S$ is compatible, then there exists a
field $K$ such that we can associate to every element in $G$ a $n\times
n$-matrix $(z^g_{ij})$ with entries in $K$ such that the conditions
(a)-(d) are satisfied.
The conditions imply in particular that the
characteristic of $K$ is $p$, that the $(z^g_{ij})$ lie in $\GL_n(K)$ and that the map $g\mapsto (z^g_{ij})$ is an injective group homomorphism. 
	
In order to show that $S$ is compatible it is, by the compactness
theorem, enough to show that every finite subset of $S$ is compatible.
For this, let $s_1,\dots, s_k$ be finitely many conditions from $S$.
Then there are only finitely many elements $g_1,\dots, g_l$ from $G$
that appear in $s_1,\dots, s_k$. The finitely generated group $\langle
g_1,\dots, g_l\rangle\subset G$ can be embedded into $\GL_n(K)$ for some
field $K$ of characteristic $p$, hence, in particular, the set of
conditions $s_1,\dots, s_k$ is compatible.
\end{proof}

\begin{rem}
	Unfortunately, Lemma \ref{embeddlem} does not allow us to say
more about the field $K$ except that it has characteristic $p$. If $G$
is infinite one can choose $K$ to have the same cardinality as $G$. But otherwise one is bound to properties of fields that can be defined by first order languages.
\end{rem}

\begin{proof}[Proof of Theorem \ref{posdim}]
Let $p$ be the characteristic of $\bbk$.
If $\bbk$ is a finite field, then the group $G$ contains subgroups of
arbitrarily large order.
Indeed, $\bbk[x]$ contains subgroups isomorphic to $(\bbZ/p\bbZ)^m$ for
all $m$.
But $\GL_n(\bbk)$ is a finite group for every $n$.
It follows that for each $n$ there exist finite subgroups of $G$ that
cannot be embedded into $\GL_n(\bbk)$. 

Suppose now that $\bbk$ is infinite. Then there exists an infinite
strictly ascending chain of finite subgroups of $\bbk$:
\[
G_1\subsetneq G_2\subsetneq \cdots
\]
where we consider $\bbk$ as a subgroup of $G$ via its non-normal factor.
For each $i$ we define a polynomial $f_i$  in $\bbk[x]\subset G$ by
\[
f_i:=\prod_{g\in G_i} (x+g).
\] 
Choose $g_1\in G_1$ and $g_i\in G_i\setminus G_{i-1}$ for all $i\geq 2$. 
We observe that $g_j f_i =f_i g_j$ if $j\leq i$,
but $g_j f_i\neq f_i g_j$ if $j>i$.  

 Assume that there exists a $n$ such that every finite subgroup of $G$
can be embedded into $\GL_n(\bbk)$.
Since every finitely generated subgroup of $G$ is finite, this implies
by Theorem \ref{embeddlem} that there exists a field $K$ of
characteristic $p$ such that $G$ can be embedded into $\GL_n(K)$.
Let $\varphi\colon G\to \GL_n(K)$ be an injective homomorphism.
We define the following Zariski-closed subsets of $\GL_n(K)$ for
each $j \ge 1$:
 \[
 Z_j:=\left\{u\in \GL_n(K) \mid u\varphi(g)=\varphi(g)u
\textrm{ for all } g \in G_j
\right\}
 \]
By construction, $\varphi(f_i)\in Z_j$ if $j\leq i$, but $\varphi(f_i)\notin Z_j$ if $j>i$. This implies that 
\[Z_1\supsetneq Z_2\supsetneq Z_3\supsetneq\cdots
\]
 is an infinite strictly descending chain of non-empty Zariski-closed subsets. But this is not possible, since the Zariski topology is Noetherian. 
\end{proof}

\begin{rem}\label{linearity}
In fact, the proof of Theorem~\ref{posdim} shows that if $\bbk$
is of positive characteristic and infinite, then the group
$\bbk[x]\rtimes\bbk$ is not linear over any field. This yields in
particular a new proof that $\Aut(\mathbb{A}^2_{\bbk})$ and
$\Cr_2(\bbk)$ are not linear in this case - a result that has previously
been proven in \cite{CerveauDeserti}, \cite{Cornulier}, and
\cite{Mathieu}. In \cite{Mathieu} it is shown moreover that
$\Aut(\mathbb{A}_\bbk^2)$ is linear if $\bbk$ is finite. To our
knowledge, it is still an open question whether  $\Cr_2(\bbk)$ is linear
if $\bbk$ is finite. 
\end{rem}

\section{Bounds from Algebraic Tori}
\label{sec:torus}

Given a finitely-generated abelian group $A$ with an action of a finite group $G$,
the \emph{symmetric rank} of $A$ with respect to $G$,
denoted $\operatorname{symrank}(A,G)$,
is the minimal size of a $G$-stable generating set of $A$.
(This nomenclature is due to MacDonald~\cite{MacDonald}.)
The following conjecture is due to the second author:

\begin{conj}[B.~Heath] \label{conj:Bailey}
If $n$ is a positive integer, then the maximum
value of $m=\operatorname{symrank}(\bbZ^n,G)$
for any finite group $G$ is given by the following table:
\begin{center}
\begin{tabular}{l|ccccccc}
$n$ & $1$ & $2$ & $3$ & $4$ & $5$ & $6$ & $\ge 7$\\
\hline
$m$ & $2$ & $6$ & $12$ & $24$ & $40$ & $72$ & $2^n$\\
\end{tabular}
\end{center}
\end{conj}

\begin{rem}
The values in Conjecture~\ref{conj:Bailey} were found by the second
author in his thesis on the representation dimension of algebraic tori
\cite{Bailey,Heath25},
where the conjecture is proved for irreducible tori
of dimension $n \le 11$ and a set of prime dimensions
that is infinite assuming Artin's conjecture on primitive roots. 
\end{rem}

In this section, we consider the representation dimensions of
automorphism groups of tori.
Recall that a \emph{split torus} is an algebraic group $T$ isomorphic
to $\bbG_{m,\bbk}^n$ for some $n$.
A \emph{torus} is an algebraic group $T$ such that $T
\times_{\Spec(\bbk)} K \simeq \bbG_{m,K}^n$ for some field extension
$K/\bbk$.

\begin{lem} \label{lem:toric2346}
Let $H$ be a finite subgroup of $\GL_n(\bbZ)$ and $d \in \{2,3,4,6\}$.
Then, over any field $\bbk$ of characteristic $\ne 2,3$,
there exists a rational torus $T$ of dimension $n$
such that $\fdim_\bbk(\Aut(T)) \ge \operatorname{symrank}((\bbZ/d\bbZ)^n,H)$.
\end{lem}

\begin{proof}
We begin by constructing a one-dimensional torus $S$ depending on $d$.
If $d=2$, then $S=\bbG_m$ and we observe that $\bbZ/2\bbZ \subseteq S(\bbk)$.
If $d=4$, then $S$ is the group scheme defined by the equation of the unit circle
$x^2+y^2=1$ in $\mathbb{A}^2_\bbk$.  The point $(0,1)$ has order $4$,
thus  $\bbZ/4\bbZ \subseteq S(\bbk)$.
If $d=3$ or $6$, then $S$ is the group scheme defined by $x^2+3y^2=1$
with a $2$-dimensional representation given by
\[
\begin{pmatrix} x & -3y \\ y & x\end{pmatrix}.
\]
In this case, the point $(\frac{1}{2},\frac{1}{2})$ has order $6$,
so $\bbZ/6\bbZ \subseteq S(\bbk)$.
Now $\bbG_m$ is clearly $\bbk$-rational,
while the other two possibilities for $S$ are seen to be $\bbk$-rational
by embedding in $\bbP^2_\bbk$ and projecting away from the identity of
the torus.

Let $T=S^n$ be our rational torus of dimension $T$.
By construction, we have $A=(\bbZ/d\bbZ)^n \subseteq T(\bbk)$.
The group $\operatorname{GL}_n(\mathbb{Z})$ acts by group
automorphisms on $T$.
Thus $H$ acts on $T(\bbk)$ by group automorphisms.
The subgroup $A$ is the subgroup of elements of exponent $d$ in
$T(\bbk)$ and thus is characteristic.
Thus, we have a finite subgroup $G=A \rtimes H \subseteq \Aut(T)$.

Suppose $\rho$ is a faithful representation of $G$ over the algebraic closure
of $\bbk$.
The restriction $\overline{\rho}$ of $\rho$ to $A$ is a direct sum of one
dimensional characters in $A^\vee := \Hom(A,\bar{\bbk}^\times)$,
which are permuted by $H$.
Let $\Omega$ be the set of these characters.
Since $\overline{\rho}$ is faithful,
$\Omega$ must be an $H$-stable generating set for $A^\vee$.
Thus, $\rho$ has degree at least
$\operatorname{symrank}((\bbZ/d\bbZ)^n,H)$ as desired.
\end{proof}

\begin{thm} \label{thm:Bailey}
Over an arbitrary field $\bbk$ of characteristic $\ne 2,3$, there exists a
rational torus $T$ of dimension $n$ such that $\fdim_\bbk(\Aut(T)) \ge m$
where $m$ is as in Conjecture~\ref{conj:Bailey}.
\end{thm}

\begin{proof}
We use Lemma~\ref{lem:toric2346} for all cases.
More precisely, for each dimension $n$, we compute
$m=\operatorname{symrank}(L/dL,W)$ where $d \in \{2,3,4,6\}$,
$L \cong \bbZ^n$, and $W$ is a subgroup of $\GL_n(\mathbb{Z})$.
In each case, $\Omega \subseteq L$ will be our $W$-stable generating
set of minimal size.

All our $L$ come from the theory of root systems where
$W$ is the associated Weyl group
(see e.g. \S{III}~of~\cite{Humphreys}).
Our choices for $\Omega$ and $L$ are essentially the
same as those used for symmetric rank of
$\mathbb{Z}^n$ in \cite{Bailey};
see also \cite{Lemire} for related work.
We summarize our choices in Table~\ref{tbl:sym_rank}.

\begin{table}[htb]
\caption{Symmetric Ranks of Root Lattices}
 \label{tbl:sym_rank}
    \centering
\begin{tabular}{c|c|c|c|c|c|}
$n$ & $L$ & $W$ & $\Omega$ & $|\Omega|$ & $d$ \\
\hline
$1$ & $\mathsf{A}_1$ root lattice & $2$ & roots & $2$ & 4 \\
$2$ & $\mathsf{G}_2$ root lattice & $D_{12}$ & short roots & $6$ & 4 \\
$3$ & $\mathsf{A}_3$ root lattice & $S_4$ & roots & $12$ & 4 \\
$4$ & $\mathsf{F}_4$ root lattice & $W(F_4)$ & short roots & $24$ & 4 \\
$5$ & $\mathsf{D}_5$ root lattice & $2^4 \rtimes S_5$ & roots & $40$ & 4 \\
$6$ & $\mathsf{E}_6$ root lattice & $W(E_6)$ & roots & $72$ & 3 \\
$\ge 7$ & $\mathsf{B}_n$ weight lattice & $2^n \rtimes S_n$ & spinors & $2^n$ & 4 \\
\end{tabular}

\end{table}

For any particular choice of $L$, $d$, and $W$,
finding the symmetric rank $\operatorname{symrank}(L/dL,W)$
is a finite problem.
Indeed, one simply enumerates all possible orbits of $W$ on
$L/dL$.  Then one finds the smallest union of orbits that generates
$L/dL$.
In the specific cases $1 \le n \le 6$ from Table~\ref{tbl:sym_rank}
it is even easier to check:
the union of all orbits of size less than $|\Omega|$ do not generate
$L/dL$, but $\Omega$ on its own does.
This was checked using \cite{sagemath}.

It remains to consider general $n$.
Here, we consider the weight lattice $L$ with $B_n$.
A convenient description of $L$ is as follows.
Let $W$ be the group $(\bbZ/2\bbZ)^n \rtimes S_n$,
which acts on $A=\bbZ^n$ by permuting basis vectors and
multiplying basis vectors by $-1$.
The lattice $L$ is the sublattice of $A \otimes_{\mathbb{Z}} \bbQ$
containing $A$ and all vectors of the form
\[
\left( \frac{a_1}{2}, \cdots, \frac{a_n}{2} \right)
\]
where $a_1,\ldots,a_n$ are all odd.
Any generating set of $L/4L$ must contain
a vector not contained in $A/4L$
and any $W$-stable generating set must contain all such vectors.
Thus $m=2^n$.
\end{proof}

\begin{rem}
The construction for $\mathsf{B}_n$ above has the following interpretation.
The odd spin group $\operatorname{Spin}_{2n+1}(\mathbb{R})$ has maximal
compact torus $(S^1)^n$ that is defined over $\mathbb{Q}$.
The normalizer of the maximal torus contains a finite subgroup
$(\mathbb{Z}/4\mathbb{Z})^n \rtimes
\left((\mathbb{Z}/2\mathbb{Z})^n \rtimes S_n \right)$
that is defined over $\mathbb{Q}$.
A minimal dimensional faithful representation of this
finite subgroup is obtained by restricting from the spin representation
of $\operatorname{Spin}_{2n+1}(\mathbb{R})$

The other finite groups are related to tori in adjoint semisimple algebraic groups,
but their minimal representation dimensions are not necessarily obtained by restriction.  For example, there is no $72$-dimensional representation of an adjoint form of $E_6$.
\end{rem}

\section{Projective spaces}
\label{sec:pgl}

In this section, we investigate the representation dimensions of finite
subgroups of of projective linear groups.  The cases of $\PGL_3(\bbk)$,
$\PGL_4(\bbk)$ will be of special interest in bounding $c_2(\bbk)$ and
$c_3(\bbk)$.

\begin{lem}
Suppose $p$ is an odd prime and $\bbk$ contains a primitive $p$th root of
unity.
Then there is a subgroup of $\PGL_p(\bbk)$ isomorphic to
$(C_p)^2 \rtimes \SL_2(\bbF_p)$.
\end{lem}

\begin{proof}
The following construction is closely related to the Weil
representation of $\SL_2(\bbF_p)$; it is certainly known to experts,
but we could not find a suitable reference.

Let $\zeta$ be the $p$th root of unity.  Let $e_0,\ldots,e_{p-1}$ denote
a basis for the vector space $W \cong \bbk^p$.
We will write $e_i$ for any integer
$i$ by considering $i \mod p$.
Let $P,D,F,V$ be matrices in $\GL(W)$ defined by their actions on $W$ as follows:
\begin{align*}
Pe_i &= e_{i-1}\\
De_i &= \zeta^i e_i\\
Fe_i &= \sum_{j=0}^{p-1} \zeta^{ij} e_j\\
Ve_i &= \sum_{j=0}^{p-1} \zeta^{v(i,j)} e_j
\end{align*}
where $v(i,j) = (i-j)(j-i-1)/2$.

First, one checks that $P^p=D^p=I$ and $PDP^{-1}D^{-1}=\zeta I$ where $I$ is the
identity matrix.
This means that the matrix group $N=\langle P, D \rangle$
has image $(C_p)^2$ in $\PGL(V)$.
Note that $N$ is an irreducible representation and the matrices
$P^aD^b$ for $a,b \in \bbZ/m\bbZ$ are a basis for the $p\times p$
matrix algebra $M_{p\times p}(\bbk)$.

Then one checks the following relations:
\[
FPF^{-1}=D^{-1}, \quad
FDF^{-1}=P, \quad
VPV^{-1}=P, \quad
VDV^{-1}=DP.
\]
(Note that this is where the fact that $p$ odd is used:
the relations are false for $p=2$).
Thus $N$ is a normal subgroup of $G=\langle P,D,V,F \rangle$.
Let $\psi : N \to \bbF_p^2$ be the group homomorphism where
$P$ and $D$ are mapped to the first and second basis vectors
respectively.
Then the induced conjugation action on $N$ gives a homomorphism
$\rho : G \to \GL_2(\bbF_p)$ where
\[
\rho(F) = \begin{pmatrix}
0 & 1 \\ -1 & 0
\end{pmatrix}
\textrm{ and }
\rho(V) = \begin{pmatrix}
1 & 1 \\ 0 & 1
\end{pmatrix} \ .
\]
It is well known that the image of $\rho$ is $\SL_2(\bbF_p)$.

Since the elements of $N$ are a generating set of $M_{p\times p}(\bbk)$
permuted (up to a scalar) by $G$, we see that the image of $G$
in $\PGL_p(\bbk)$ is $(C_p)^2 \rtimes \SL_2(\bbF_p)$ as desired.
\end{proof}

\begin{prop} \label{prop:pglp}
Suppose $p$ is an odd prime.
If $\bbk$ is a field of characteristic $0$ that contains a primitive $p$th
root of unity then $\fdim_\bbk(\PGL_p(\bbk))=p^2-1$.
\end{prop}

\begin{proof}
It is clear that $p^2-1$ is an upper bound by taking the adjoint
representation of $\PGL_{p,\bbk}$.

By the above lemma,
there is a subgroup $G$ of $\PGL_p(\bbk)$
that is isomorphic to $(C_p)^2 \rtimes \SL_2(\bbF_p)$.
Let $N \cong (C_p)^2$ be the evident normal subgroup of $G$
and let $H$ be the factor group $G/N \cong \SL_2(\bbF_p)$.
Observe that $H$ acts transitively on the non-trivial elements of $N$.

Suppose $\rho$ is a faithful representation of $G$.
By Proposition~24~of~\cite{SerreLinear}, irreducible representations of semidirect
products with abelian normal subgroups are always constructed as
follows.  Take a character $\chi$ of $N$, take a representation $\tau$
of the stabilizer $S$ of the action of $H$ on $\chi$, then induce the
representation $\chi \otimes \tau$ of $NS$ up to $G$.

Since $H$ acts transitively on the non-trivial characters of $N$ in our
situation, any faithful representation of $G$ must have an irreducible
subrepresentation with
representation dimension equal or greater to the number of non-trivial
characters of $N$.
This gives the lower bound of $p^2-1$.
\end{proof}

\begin{rem}
Already for $n=5$, the upper bound from Proposition~\ref{prop:pglp} is
$24$, which is exactly equal to the lower bound from
Theorem~\ref{thm:Bailey}.
Thus, projective spaces are only potentially relevant for $c_n(\bbk)$
if $n \le 4$.
\end{rem}

\begin{prop} \label{prop:pgl3_no_root}
If $\bbk$ has characteristic $0$ and does not contain a primitive third
root of unity then $\fdim_\bbk(\PGL_3(\bbk)) \le 6$.
\end{prop}

\begin{proof}
Let $G$ be a finite subgroup of $\PGL_3(\bbk)$.
Let $\widetilde{G}$ be the preimage of $G$ in $\GL_3(\bbk)$
and let $\rho$ be the corresponding representation.

If $\rho$ is reducible, then $\rho = \sigma \oplus \tau$
where $\sigma$ has degree $1$ and $\tau$ has degree $2$.
Note that the dual representation $\sigma^\vee$ is
simply the homomorphism $\widetilde{G} \to k^\times$
given by $\sigma^\vee(g) = \sigma(g)^{-1}$.
The representation $\sigma^\vee \otimes \tau$ factors
through the homomorphism $\widetilde{G} \to G$ and provides a
faithful two-dimensional representation of $G$.

Now suppose that $\rho$ is irreducible.
Suppose moreover that $\rho$ is \emph{imprimitive};
this means $\rho$ is induced from a representation $\sigma$ of a proper subgroup $H$.
Observe that $\sigma$ must be $1$-dimensional.
Thus $\rho$ is monomial and $\widetilde{G}$ is a subgroup of
$\bbG_m^3 \rtimes S_3$.
We claim that there is a group scheme morphism
\[
\psi : \bbG_m^3 \rtimes S_3 \to \bbG_m^6 \rtimes S_6 \subset \GL_6
\]
that has kernel the diagonal $\bbG_m$.
Indeed, on diagonal matrices we define
\[
\psi(\lambda_1,\lambda_2,\lambda_3) =
(\lambda^{}_1\lambda_3^{-1},
\lambda^{}_2\lambda_3^{-1},
\lambda^{}_1\lambda_2^{-1},
\lambda^{}_3\lambda_2^{-1},
\lambda^{}_2\lambda_1^{-1},
\lambda^{}_3\lambda_1^{-1})
\]
and the permutation matrices are then determined.
Thus $\psi(\widetilde{G})=G$ and the group $G$
has a representation of degree at most $6$.

It remains to consider the \emph{primitive} subgroups
of $\PGL_3(\bbk)$.  Over $\bbC$, these were classified by Blichfeldt in
\cite{Blichfeldt}.
Up to conjugacy, the primitive subgroups of $\PGL_3(\bbC)$ are
\begin{enumerate}
\item the alternating group $A_5$ on $5$ letters,
\item the alternating group $A_6$ on $6$ letters,
\item the simple group $\PSL_2(\bbF_7) \cong \PSL_3(\bbF_2)$ of order $168$,
\item the Hessian group $C_3^2 \rtimes \SL_2(\bbF_3)$, and
\item subgroups $C_3^2 \rtimes \bbZ/4\bbZ$
and $C_3^2 \rtimes Q_8$ of the Hessian group.
\end{enumerate}

Each of these groups may or may not exist in $\PGL_3(\bbk)$.
Note that $A_5$ and $A_6$ have faithful representations of dimension $4$
and $5$ respectively that are defined over $\bbQ$.
The group $G=\PSL_3(\bbF_2)$ acts faithfully on the $7$ points of
$\bbP^2(\bbF_2)$, so there is a $7$-dimensional permutation
representation of $G$.  The trivial representation is a direct
summand, so there is a $6$-dimensional faithful representation of
$G$ over $\bbQ$.
Regardless of $\bbk$, these groups have representation dimensions at
most $6$.

It remains to consider the Hessian group and its subgroups.
We will show that no subgroup of $\PGL_3(\bbk)$ contains $C_3^2$
and thus none of the primitive subgroups of the Hessian are realizable
over $\bbk$.

Suppose otherwise; that $a,b \in \PGL_3(\bbk)$ generate a subgroup $G$
isomorphic to $C_3^2$.
Let $A,B \in \GL_3(\bbk)$ be lifts of $a,b$ respectively.
Over the algebraic closure, we may choose scalar multiples
$\widetilde{A},\widetilde{B}$ of $A,B$
that lie in $\SL_3(\overline{\bbk})$ and generate a finite subgroup.

The characteristic polynomial of $A$ is
$x^3-\lambda$ where $\lambda \in \bbk$.  Since $\bbk$ does not contain a
non-trivial third root of unity, the eigenvalues of $A$ must be
distinct.  The eigenvalues of $\widetilde{A}$ must therefore be
$1,\zeta_3,\zeta_3^2$.
The same is true of $B$ and $\widetilde{B}$.

Observe that the commutator $ABA^{-1}B^{-1}=\lambda I_3$
is a scalar matrix where $\lambda \in \bbk$ is independent of choice of
lift.
Since $\widetilde{A},\widetilde{B}$ have the same commutator, which must
be in $\operatorname{SL}_3(\bbk)$ and have finite order,
we have $\lambda=1$.
We conclude that $\widetilde{A},\widetilde{B}$ commute
and are therefore simultaneously diagonalizable.
Up to scalar multiples, there are exactly two diagonal
matrices of order $3$ with distinct eigenvalues.
Thus $\widetilde{A}$,$\widetilde{B}$ generate a cyclic group,
which contradicts that $G \cong C_3^2$.
\end{proof}

\begin{prop} \label{prop:pgl4}
If $\bbk$ has characteristic $0$ and
contains a primitive $4$th root of unity, then
$\fdim_\bbk(\PGL_4(\bbk))=15$.
\end{prop}

\begin{proof}
It is clear that $15$ is an upper bound by taking the adjoint
representation of $\PGL_{4,\bbk}$.
Consider the pseudoreflection group $G_{29}$
from the Shephard-Todd classification, which is defined over the field
$\bbQ(i)$ (see \S{8} of \cite{Benard}).
The group $G_{29}$ is a subgroup of $\GL_4(\bbk)$ where the
image $H \subseteq \PGL_4(\bbk)$ is a quotient by $\bbZ/4\bbZ$.
From Table~V~of~\cite{Benard}, we observe that $H$ has representation dimension $15$.
\end{proof}

\section{Weighted Projective Spaces}
\label{sec:wps}

We recall the theory of weighted projective spaces; the standard
reference is \cite{Dol82Weighted}
(see also Appendix~A.2~of~\cite{PSJordan3}).

Given a finite set of positive integers $q_1,\ldots,q_s$
(called \emph{weights}),
let $S=\bbk[x_1,\ldots,x_s]$ be the graded polynomial ring
where $\deg(x_i)=q_i$.
The \emph{weighted projective space} $\mathbb{P}(q_1,\ldots,q_s)$
is the projective variety $\operatorname{Proj}(S)$;
the ring $S$ is the \emph{Cox ring} of the space.
Equivalently,
$\mathbb{P}(q_1,\ldots,q_s)$ is the universal geometric quotient of
$\mathbb{A}^s_\bbk \setminus \{ 0 \}$ by $\bbG_m$
where the action is given by
\[
t \cdot (x_1,\ldots,x_s) \mapsto (t^{q_1}x_1,\ldots,t^{q_s}x_s)\ .
\]

The ordering of $q_1,\ldots,q_s$ is unimportant.  Every weighted
projective space is isomorphic to a \emph{well-formed} weighted projective space
where every subset $Q'$ of $s-1$ weights has $\gcd(Q')=1$.
A convenient shorthand
is to write $\mathbb{P}(d_1^{m_1}:\ldots:d_r^{m_r})$ for the weights
where each $d_i$ occurs $m_i$ times.

The automorphism group schemes of weighted projective spaces are well understood:

\begin{lem} \label{lem:wp_aut}
Suppose $X=\mathbb{P}(d_1^{m_1}:\ldots:d_r^{m_r})$ with Cox ring $S$.
Then
\[
\sheafAut(X) \cong U \rtimes \left(\left(\prod_{i=1}^r
\GL_{m_i}\right)\middle/\bbG_{m}\right)
\]
where $U$ is unipotent and $\bbG_m$ acts via
\[ t \mapsto (t^{d_1} I_{m_1}, \ldots, t^{d_r} I_{m_r}) \ . \]
\end{lem}

\begin{proof}
This follows from the more general description for automorphism groups
of toric varieties from Section~4~of~\cite{Cox}.
Details for the special case of weighted projective spaces can be found
in Proposition~A.2.5~of~\cite{PSJordan3}.
Both these sources work over algebraically closed fields of
characteristic $0$, but
the result holds over any field; see \cite{LiendoLA}.
\end{proof}

In general, the representation dimension of the group scheme
$\sheafAut(X)$ (or its reductive quotient)
is a delicate function of the weights.
We will only need a relatively simple case in this paper:

\begin{lem} \label{lem:wp_rdim}
Suppose $\bbk$ has characteristic $0$.
If $X$ is the weighted projective space $\mathbb{P}(1^2:n^m)$, then
\[
\fdim_{\bbk}( \Aut( X ) ) \le
\begin{cases}
2m & \textrm{if $n$ is odd,}\\
3m & \textrm{if $n$ is even.}
\end{cases}
\]
\end{lem}

\begin{proof}
Since $\cha(\bbk)=0$, any finite subgroup of $\Aut(X)$
maps isomorphically to the reductive quotient.
Thus, it suffices to bound the representation dimension of the
reductive quotient $\sheafAut(X)/U$,

If $n$ is odd, then let
$\phi_n : \GL_2 \times \GL_m \to \GL_{2m}$
be the group scheme homomorphism given by
\[
\phi_n(A,B) := \det(A)^{\frac{-n-1}{2}}\left(A\otimes B\right).
\]
If $n$ is even, then let
$\phi_n : \GL_2 \times \GL_m \to \GL_{3m}$
be the group scheme homomorphism given by
\[
\phi_n(A,B) := \det(A)^{-\frac{n}{2}-1}\left(\sigma(A)\otimes B\right).
\]
where $\sigma : \GL_2 \to \GL_3$ is the symmetric square.
In both cases, the kernel is of the form $(tI_2,t^nI_m)$.
Thus the result follows by the first isomorphism theorem
in view of Lemma~\ref{lem:wp_aut}.
\end{proof}

Recall that a (smooth) \emph{Fano variety} $X$ is a variety with ample
anticanonical divisor $-K_X$.  A \emph{del Pezzo surface} is a
two-dimensional Fano variety.
Given a Fano variety, we have an associated \emph{anticanonical ring}
\[
R(X,-K_X) := \bigoplus_{n \ge 0} H^0(X,-nK_X),
\]
which is a finitely generated graded ring.
A choice of minimal generating set $x_1,\ldots,x_s$ for $R(X,-K_X)$
gives rise to an embedding
$X \hookrightarrow \mathbb{P}(d_1,\ldots,d_s)$
where $d_i = \deg(x_i)$.

\begin{lem} \label{lem:anticanonical}
Let $X$ be a smooth Fano variety over a field $\bbk$ of characteristic
$0$.
Then $\fdim_\bbk(\Aut(X)) \le s$ where
$s$ is the size of a generating set of $R(X,-K_X)$.
\end{lem}

\begin{proof}
Let $G$ be a finite group with a faithful action on $X$.
The canonical bundle $\omega_X$ has a canonical linearization of $G$
(see Proposition~2.11~of \cite{BCDP}).
This means that $G$ has a faithful action of $G$ on the Cox ring $S$
via graded ring homomorphisms.
Since $G$ is a finite group, $G$ is a subgroup of the reductive part
$\prod_{i=1}^r \GL_{m_i}(\bbk)$ of the group of graded automorphisms
from Lemma~\ref{lem:wp_aut}.
We conclude $\rdim_\bbk(G) \le s$ as desired.
\end{proof}

\begin{rem}
Lemma~\ref{lem:anticanonical} cannot be extended to all
section rings $R(X,H)$.
Indeed, a simple counterexample is the standard graded polynomial ring
$\bigoplus_{n \ge 0} H^0(X,\mathcal{O}(n))$ for $\mathbb{P}^1$.
Here, the section ring has $2$ generators, but $\fdim_\bbk(\PGL_2)$ can be $3$.
The argument fails because $\mathcal{O}(1)$ may not be linearizable,
while $\mathcal{O}(2)$ can always be linearized since it is
anticanonical.
\end{rem}

\begin{cor} \label{cor:delPezzo}
If $X$ is a del Pezzo surface of degree $d$, then
\[
\fdim_\bbk(\Aut(X)) \le
\begin{cases}
4 & \textrm{if $d=1$ or $d=2$,}\\
d+1 & \textrm{otherwise.}
\end{cases}
\]
\end{cor}

\begin{proof}
If $d \ge 3$, then $-K_X$ is very ample, so
the anticanonical ring has
$\dim H^0(X,-K_X) = d+1$ generators in degree $1$.
If $d=2$ (resp. $d=1$), then the anticanonical ring induces embeddings
of $X$ into $\mathbb{P}(1,1,1,2)$ (resp. $\mathbb{P}(1,1,2,3)$);
see \cite[\S{6.6,7}]{DolIsk}.
The bounds now follow from
Lemma~\ref{lem:anticanonical}.
\end{proof}

\begin{rem}
The bounds in Corollary~\ref{cor:delPezzo} are certainly not sharp.
This is evident even in the case where $X=\mathbb{P}^2$.
Proposition~\ref{prop:del_pezzo} below improves on some of these bounds.
\end{rem}

The next corollary will not be used elsewhere in the paper, but
may be of independent interest.  A more detailed analysis of smooth Fano
$3$-folds will certainly yield better bounds.

\begin{cor} \label{cor:fano3}
If $X$ is a smooth complex Fano $3$-fold with $-K_X$ very ample,
then $\fdim_\bbC(\Aut(X)) \le 35$. 
\end{cor}

\begin{proof}
Lemma~\ref{lem:anticanonical} tells us that
$\fdim_\bbC(\Aut(X)) \le \dim H^0(X,-K_X)$.
By Hirzebruch-Riemann-Roch, we know that
\[
\dim H^0(X,-K_X) = (-K_X)^3/2+3.
\]
Iskovskikh proved that $(-K_X)^3 \le 64$ for all
smooth Fano threefolds.  The result follows immediately.
\end{proof}

\section{Plane Cremona Groups}
\label{sec:cr2}

In this section, we prove the characteristic $0$ part of Theorem~\ref{thm:c2}.

In order to understand finite subgroups of the plane Cremona group,
the standard approach reduces the question to minimal rational surfaces.

\begin{thm}[Manin, Iskovskikh] \label{thm:ManIsk}
Suppose $\bbk$ is a field of characteristic $0$ and
$G$ is a finite subgroup of $\Cr_2(\bbk)$.
There exists a rational surface $X$ over $\bbk$ with a faithful action of $G$
such that either
\begin{enumerate}
\item $X$ is a del Pezzo surface with $\Pic(X)^G \cong \bbZ$, or
\item $X$ admits a $G$-conic bundle structure
$X \to \mathbb{P}^1$ with $\Pic(X)^G \cong \bbZ^2$
\end{enumerate}
\end{thm}

\begin{proof}
This is essentially the classification of minimal geometrically-rational
$G$-surfaces, which is due to Manin~\cite{Manin} and
Iskovskikh~\cite{Iskovskikh}.
The original papers handled \emph{either} the case of a non-closed field
or a non-trivial $G$-action. 
We refer to {\cite[Theorem~5 and Remark~1]{DolIsk2}} for the
mixed case of a $G$-action over a non-closed field.
Note that in general the base of the conic bundle
may only be a smooth genus $0$ curve, but in our case $X$ is rational
so we may assume it is $\mathbb{P}^1$. 
\end{proof}

First we consider del Pezzo surfaces.  We begin with quadric surfaces:

\begin{prop} \label{prop:P1P1}
If $X$ is a $\bbk$-form of $\bbP^1 \times \bbP^1$,
then $\rdim_\bbk(\sheafAut(X)) \le 6$.
\end{prop}

\begin{proof}
In view of Lemma~\ref{lem:forms}, it
suffices to show that $\rdim_\bbk(\sheafAut(X)) \le 6$
for $X = \bbP^1 \times \bbP^1$.
Recall that we have an isomorphism of group schemes
\[
\sheafAut(\bbP^1 \times \bbP^1) \cong \PGL_{2,\bbk}^2 \rtimes C_2
\]
where $C_2$ interchanges the two copies of $\PGL_{2,\bbk}$.
There is an injective homomorphism of algebraic groups
$\PGL_{2,\bbk} \to \GL_{3,\bbk}$.
Thus we have an embedding
$\sheafAut(X) \hookrightarrow \GL_{6,\bbk}$
where each factor of $\PGL_{2,\bbk}$ is embedded block diagonally
and $C_2$ acts by interchanging the factors.
\end{proof}

\begin{prop} \label{prop:del_pezzo}
If $X$ is a rational del Pezzo surface not isomorphic to $\bbP^2$,
then $\fdim_\bbk(\Aut(X)) \le 6$.
\end{prop}

\begin{proof}
Recall that when $\bbk$ is algebraically closed,
then a del Pezzo surface $X$ is either $\bbP^1 \times \bbP^1$
or is $\bbP^2$ blown up at $9-d$ points where $d = K_X^2$
is the degree of the del Pezzo surface.
The case where $X$ is a $k$-form of $\bbP^1 \times \bbP^1$
was considered in Proposition~\ref{prop:P1P1}.
Thus, we may assume that $X$ is a form of $\bbP^2$ blown up at $9-d$
points.

When $d=9$, the forms of $\bbP^2$ are Severi-Brauer surfaces;
these are never rational.
When $d = 8$ or $d=7$, there is always a canonical exceptional curve
defined over $\bbk$ which can be $\Aut(X)$-equivariantly blown down;
thus these cases give embeddings of $\sheafAut(X) \to \PGL_3(\bbk)$,
which fix a $\bbk$-point.  Thus $\Aut(X)$ corresponds to the image of a
reducible representation of degree $3$, which bounds
$\fdim_\bbk(\Aut(X))$ as desired.
The remaining cases are of degree $d \le 6$,
which follow from Corollary~\ref{cor:delPezzo}.
\end{proof}

Suppose $\bbk$ is algebraically closed of characteristic $0$.
A \emph{conic bundle surface} is a morphism $\phi: X \to C$ where $X$ is a smooth
projective surface, $C$ is a smooth curve, $\phi$ has at most finitely
many singular fibers, the smooth fibers are
isomorphic to $\mathbb{P}^1$, and the singular fibers are a
union of two smooth rational $(-1)$-curves meeting at a single point.
If $\phi$ has no singular fibers, then $X$ is a \emph{ruled surface}.
We say $\phi$ is an \emph{exceptional conic bundle} if there exists an
integer $n$ such that $\phi$ has exactly $2n$ singular fibers and two
sections of self-intersection $-n$ (see Definition~2.1~of~\cite{Fong}).

In the case where $\bbk$ is not necessarily closed,
a \emph{conic bundle surface} is a morphism $\phi: X \to C$ such that
$\overline{\phi} : \overline{X} \to \overline{C}$
is a conic bundle over the algebraic closure.
We are interested exclusively in \emph{rational} conic bundle surfaces
so we may assume that $C \cong \mathbb{P}^1$.
In this case, the rational ruled surfaces $X$ are precisely the
Hirzebruch surfaces $\mathbf{F}_n$ for $n \ge 0$.

\begin{prop} \label{prop:Hirzebruch}
If $\bbk$ has characteristic $0$, $n \ge 2$ is an integer,
and $X$ is a Hirzebruch surface $\mathbf{F}_n$,
then $\fdim_\bbk(\Aut(X)) \le 3$. 
\end{prop}

\begin{proof}
Since $n \ge 1$, there is a canonical section $E_0$ of $\phi$
with self-intersection $-n$.
Blowing down $E_0$ gives a map from $\mathbf{F}_n$ to
the weighted projective space $\bbP(1,1,n)$.
The result now follows by Lemma~\ref{lem:wp_rdim}.
\end{proof}

The geometric Picard group $\Pic(\overline{X})$ is generated by a
section $E$ with self-intersection $-n$, a general fiber $F$,
and a choice of irreducible component $R_i$ in each of the $k$ singular
fibers of
$\phi$.  Since $E$ intersects each fiber only once, we may choose the
components $R_1,\ldots,R_k$ such that that $E \cap R_i = \emptyset$.
In this basis, the intersection theory is determined by the following:
\begin{align*}
E^2 &= -n & F^2&=0 & E\cdot F&=1\\
E \cdot R_i &= 0 & F \cdot R_i &= 0 & 
R_i \cdot R_j &= - \delta_{ij}
\end{align*}
and the canonical bundle is given by
\[
K_X = -2E + (-2+n)F + \sum_{i=1}^k R_i .
\]

Let $\sheafAut_C(X)$ be the subgroup scheme of $\sheafAut(X)$
that leaves invariant the morphism $\phi : X \to C$.
There is a natural map $\rho: \Aut(\bar{X}) \to \Aut(\Pic(\bar{X}))$
induced by $g \mapsto (g^\ast)^{-1}$.
The subgroup $\rho(\Aut_C(X))$ leaves invariant $K_X$ and $F$
in $\Pic(X)$.
Thus, the action on $\Pic(\bar{X})$ is determined by the images of the $R_i$.
The irreducible components of the singular fibers
\[
R_1,\ldots,R_k,F-R_1,\ldots,F-R_k
\] must be permuted amongst themselves
and the singular fibers $R_i \cup (F-R_i)$ must also be permuted.
We conclude that
\begin{equation} \label{eq:conicPic}
\rho(\Aut_C(X)) \subseteq C_2^k \rtimes S_k
\end{equation}
where $S_k$ corresponds to permutations of the singular fibers
and $C_2^k$ corresponds to interchanging components within singular
fibers.

\begin{prop} \label{prop:non_exceptional_conic}
If $\bbk$ has characteristic $0$, the map $\phi : X \to \mathbb{P}^1$ is a conic
bundle, and $\rho$ is injective,
then $\fdim_\bbk(\Aut_{\mathbb{P}^1}(X)) \le 6$.
\end{prop}

\begin{proof}
Let $G$ be a subgroup of $\Aut_{\mathbb{P}^1}(X)$.
As in Theorem~5.7~of~\cite{DolIsk}, from \eqref{eq:conicPic}
we obtain an exact sequence
\[
1 \to N \to G \to P \to 1
\]
where $P$ is the image of $G$ in $S_k$
and $N$ is the restriction to $C_2^k$.
The map $\phi : X \to \mathbb{P}^1$ induces a map $G \to \PGL_2(\bbk)$,
for which $N$ is the kernel and $P$ is the image.
Thus, $P$ is a subgroup of $\PGL_2(\bbk)$.
Since $N$ acts faithfully on the geometric fiber of $\phi$,
either $N \cong \bbZ/2$ or $N \cong (\bbZ/2\bbZ)^2$.
At this point, it is a tedious but mostly elementary exercise
in group and representation theory.
We relegate it to Appendix~\ref{sec:tedious_group_theory}.
\end{proof}

If $\rho$ is \emph{not} injective and $\phi$ is not a ruled surface,
then $X$ is an exceptional conic bundle
by Lemma~2.16~of~\cite{Fong} or Proposition~5.5~of~\cite{DolIsk}.
(Both these references work over algebraically closed fields,
but being an exceptional conic bundle is a geometric property.)

\begin{prop} \label{prop:exceptional_conic}
If $\bbk$ has characteristic $0$ and $X \to \mathbb{P}^1$
is a rational exceptional conic bundle,
then $\fdim_\bbk(\Aut(X)) \le 6$.
\end{prop}

\begin{proof}
By definition, over the algebraic closure,
there exist exactly $2n$ singular fibers
and exactly two sections $E_0$ and $E_\infty$
with self-intersection $-n$.
From the ``Third Construction'' in \S{5}~of~\cite{DolIsk},
the two sections are disjoint and blowing them down
gives a morphism $\overline{X} \to \bbP(1,1,n,n)$.
Since the pair of sections are uniquely defined, their union is defined
over $\bbk$.  Thus we have a canonical morphism $X \to P$
where $P$ is a $\bbk$-form of $\bbP(1,1,n,n)$.
Since $X$ is rational, $P \cong \bbP(1,1,n,n)$.
The result now follows from Lemma~\ref{lem:wp_rdim}.
\end{proof}

We now have a complete description of $c_2(\bbk)$ for all fields:

\begin{proof}[Proof of Theorem~\ref{thm:c2}]
The lower bound of $6$ for all fields of characteristic $0$
follows from Theorem~\ref{thm:Bailey}.
Now we consider upper bounds.
By Theorem~\ref{thm:ManIsk}, it suffices to bound $\fdim_\bbk(\Aut(X))$
for a rational del Pezzo surface $X$ or a conic bundle surface $X$.

Suppose $X$ is a rational del Pezzo surface.
If $X$ is not isomorphic to $\mathbb{P}^2$, then
$\fdim_\bbk(\Aut(X)) \le 6$ by Proposition~\ref{prop:del_pezzo}.
The remaining case is $X \cong \mathbb{P}^2$, which is the only case
where the base field is relevant.
If $\bbk$ contains a root of unity, then
$\fdim_\bbk(\PGL_3(\bbk))=8$ by Proposition~\ref{prop:pglp}.
Otherwise, $\fdim_\bbk(\PGL_3(\bbk)) \le 6$ by
Proposition~\ref{prop:pgl3_no_root}.

Now suppose $X$ has a conic bundle structure.
If $X$ has no singular fibers, then it is a rational ruled surface:
a Hirzebruch surface $\mathbf{F}_n$.
In the cases where $n \le 1$, the surface $X$ is also a del Pezzo surface
and so was handled above.
Otherwise, we appeal to Proposition~\ref{prop:Hirzebruch}.
The remaining cases are handled by
Proposition~\ref{prop:exceptional_conic}, and
Proposition~\ref{prop:non_exceptional_conic}.
\end{proof}

\section{Higher Cremona Groups}
\label{sec:cr_higher}

Here we prove Theorem~\ref{thm:finiteness}.

\begin{proof}
First, we consider the case where $\bbk$ has all roots of unity.
We claim it suffices to assume $\bbk=\bbC$.
Indeed, if $G$ is a finite group in $\GL_n(\bbC)$, then there exists a
basis such that the entries of the matrices defining the representation
lie in the field generated by the traces of these matrices
\cite[Proposition 33]{SerreLinear}.
These traces are all sums of roots of unity since $G$ is finite.

In \cite[Theorem~4.2]{PSJordan} it is shown that there exists a constant
$K(n)$ such that every finite group $G$ acting biregularly on a rational
smooth projective variety of dimension $n$ over the field $\bbC$
contains a subgroup $H$ of index at most $K(n)$ that fixes a point $p\in
X$. It is well known that the representation of $H$ on the tangent space
$T_p X$ is faithful (see for example \cite[Lemma~4]{Popov}). In
particular, $H$ is isomorphic to a subgroup of $\GL_n(\bbC)$ and hence
$G$ can be embedded into $\GL_N(\bbC)$, where $N=K(n)n$, and we obtain
$c_n(\bbC)\leq N$. 

Next, we consider the case where $\bbk$ is a finitely generated extension
of $\bbQ$.
This is a consequence of the boundedness of finite subgroups of
$\Cr_n(\bbk)$ proven in \cite[Corollary 1.5]{PSJordan} and \cite{Birkar}.
\end{proof}

\begin{rem}
It is likely that Theorem~\ref{thm:finiteness} applies
for all fields of characteristic $0$.
However, this would likely require careful extension of the arguments
from \cite{PSJordan} and is beyond the scope of this paper.
\end{rem}

\begin{thm} \label{thm:c3_bounded}
Suppose $\bbk$ is a field of characteristic $0$ containing all roots of
unity.
Then $15 \le c_3(\bbk) \le 62208$.
\end{thm}

\begin{proof}
Since $i \in \bbk$, the lower bound follows from
Proposition~\ref{prop:pgl4}.

From \cite{PSJordan3}, any finite subgroup $G$ of $\Cr_3(\bbk)$
has an abelian subgroup $A$ of index $\le 10368$.
From \cite{Prok2Group}, the maximal rank of an abelian subgroup $A$
of $\Cr_3(\bbk)$ is $6$.
Since $\bbk$ contains all roots of unity, there exists
a faithful representation of $A$ of dimension $\le 6$.
The induced representation gives a faithful representation of $G$
of dimension $\le 6 \times 10368$.
\end{proof}

\begin{rem}
A more careful application of the ideas in \cite{PSJordan3} will most
likely show that that the upper bound in Theorem~\ref{thm:c3_bounded}
is too high.  Indeed, the bound of Corollary~\ref{cor:fano3}
for smooth Fano varieties is much smaller than this.
However, as with surfaces in this paper, the conic bundle case seems
challenging since representation dimensions of group extensions are
difficult to control.
\end{rem}

\appendix
\section{Extensions of polyhedral groups}
\label{sec:tedious_group_theory}

The purpose of this section is to prove the following statement:

\begin{thm} \label{thm:dihedral_ext}
Let $\bbk$ be a field of characteristic $0$.
Suppose $G$ sits inside an exact sequence
\[
1 \to N \to G \to P \to 1
\]
where $N \cong C_2$ or $N \cong C_2^2$,
and $P$ is a finite subgroup of $\PGL_2(\bbk)$.
Then $\rdim_{\bbk}(G) \le 6$.
\end{thm}

Recall that when $\bbk$ is algebraically closed of characteristic $0$,
the subgroups of $\PGL_2(\bbk)$ are either cyclic, dihedral, or
one of the groups $A_4$, $S_4$ or $A_5$ (see for instance \cite{Beauville}).
First, we consider the case where $P$ is cyclic or dihedral.

\begin{lem}
Let $\bbk$ be a field of characteristic $0$.
Suppose $G$ sits inside an exact sequence
\begin{equation} \label{eq:appendix_ext}
1 \to N \to G \to P \to 1
\end{equation}
where $N \cong C_2$ or $N \cong C_2^2$,
and $P$ is a cyclic or dihedral subgroup of $\PGL_2(\bbk)$.
Then $\rdim_{\bbk}(G) \le 6$.
\end{lem}

\begin{proof}
Recall that $\Aut(C_2)=1$
and $\Aut(C_2^2) \cong S_3$.
Also, observe that if $C_n \subseteq \PGL_2(\bbk)$
then $\rdim_\bbk(C_{2n}) \le 2$
and $\rdim_\bbk(D_{2n}) \le 2$.

Suppose $P \cong C_n$ for some positive integer $n$.
If $N$ is central then $P$ is one of the following possibilites:
\[
C_n \times C_2,\ C_{2n},\ C_n \times C_2^2, C_{2n} \times C_2,
\]
which may not be distinct depending on $n$.
In all cases we obtain $\rdim_\bbk(G) \le 3$.

Suppose $N$ is \emph{not} central.
In this case, $N \cong C_2^2$ and the image $I$ of $\pi : P \to S_3$
is either $C_2$ or $C_3$.
If the image is $C_2$, then let $H$ be the kernel of the map from
$G \to C_2$.
Now $H$ is a central extension of a cyclic group.
Thus, by taking induced representations
$\rdim_\bbk(G) \le 2\rdim_\bbk(H) \le 6$.

In the case of $I \cong C_3$,
let $g \in G$ map to a generator of $P$.
Suppose the image of $g$ has order $n > 1$ in $P$.
In other words, $g^n \in N$.
If $g^n \ne 1$, then $g^n$ is not invariant under conjugation by $g$,
which is absurd.
Thus we have a splitting $G=N \rtimes P$.
Let $K$ be the subgroup of $P$ of order $C_{n/3}$ embedded into $G$.
Thus $G/K \cong A_4$.
Since $N$ and $K$ are complementary,
we have an injective morphism $G \to P \times A_4$.
Thus $\rdim_\bbk(G) \le \rdim_\bbk(P) + \rdim_\bbk(A_4) \le 6$.

Suppose $P \cong D_{2n}$.
Let $H$ be the preimage of the normal subgroup $C_n$ in $P$.
If $H$ is a central extension as above, then,
by taking induced representations, we obtain
$\rdim_\bbk(G) \le 2\rdim_\bbk(H) \le 6$.

If $H$ is not a central extension, then $N \cong C_2^2$
and the map $\pi : P \to S_3$ is surjective
with restriction $\pi(H) \cong C_3$.
Similar to above, we have a splitting $G = N \rtimes P$ with
$K \subseteq H$ of order $C_{n/3}$ and $G/K \cong S_4$.
We have an injective morphism $G \to P \times S_4$.
Thus $\rdim_\bbk(G) \le \rdim_\bbk(P) + \rdim_\bbk(S_4) \le 6$.
\end{proof}

The remaining cases of $A_4$, $S_4$ and $A_5$ are the most complicated,
however it is a finite problem.
In order to prove Theorem~\ref{thm:dihedral_ext},
we explicitly enumerate all possible groups $G$ fitting into
an extension \eqref{eq:appendix_ext} in these cases.
The full list can be found in Table~\ref{tbl:poly_exts} below.
We then find a ``worst case'' upper bound for $\rdim_{\bbk}(G)$
under the weakest assumptions on $\bbk$
such that $P \subseteq \PGL_2(\bbk)$
(namely, $-1$ is a sum of squares for all three cases
and $\sqrt{5} \in \bbk$ if $P=A_5$).
The enumeration of possible extensions was computed by consulting the
Small Groups Database in GAP~\cite{Gap4}.
The representation dimensions were found by
inspection of the character tables of the groups, with additional
arithmetic information computed via the wedderga
package~\cite{wedderga}.

For the sake of completeness, we also prove the theorem totally by hand
without the use of computer assistance.  This constitutes the remainder
of the appendix.

The groups $A_4$, $S_4$ and $A_5$ are also known as the tetrahedral,
octahedral, and icosahedral groups, respectively.
These three polyhedral groups do not have $2$-dimensional
representations, and any preimage of these groups in $\PGL_2(\bbk)$
in $\GL_2(\bbk)$ must contain the binary tetrahedral,
binary octahedral, or binary icosahedral group.
These are all central stem extensions of the corresponding polyhedral
group by $C_2$.

We recall some facts about projective representations of symmetric
groups (see, e.g., \cite{stembridge}).
We denote the binary tetrahedral group by $\widetilde{A_4}$
and the binary icosahedral group by $\widetilde{A_5}$;
these two groups are the unique non-split extensions of $A_4$ and $A_5$
by $C_2$.
We denote the binary octahedral group by $\widetilde{S_4}_+$ since there
is another non-split central stem extension,
which we denote by $\widetilde{S_4}_-$.
The two groups can be distinguished by the fact that transpositions have
preimages of order $2$ in $\widetilde{S_4}_+$, and preimages of order
$4$ in $\widetilde{S_4}_-$.

We have the following convenient exceptional isomorphisms:
\begin{center}
\begin{tabular}{ccc}
$\begin{aligned}
A_4 &\cong \PSL_2(3) \\
\widetilde{A_4} &\cong \SL_2(3)
\end{aligned}$ &
$\begin{aligned}
S_4 &\cong \PGL_2(3) \\
\widetilde{S_4}_+ &\cong \GL_2(3)
\end{aligned}$ &
$\begin{aligned}
A_5 &\cong \PSL_2(4) \cong \PSL_2(5) \\
\widetilde{A_5} &\cong \SL_2(5) 
\end{aligned}$
\end{tabular}
\end{center}

We remark that $\widetilde{S_4}_+$ has an outer automorphism
given by the map $M \mapsto \det(M)M$ where $M$ is a matrix
in $\GL_2(3) \cong \widetilde{S_4}_+$.

Now we consider the structure of $\GL_2(\bbZ/4\bbZ)$.
First, observe that
\[
H= \left\langle
\begin{pmatrix} 0 & 1\\1 & 0 \end{pmatrix},\
\begin{pmatrix} 3 & 3\\1 & 0 \end{pmatrix}
\right\rangle
\]
is a subgroup isomorphic to $S_3$.
Next, observe that the subgroups
\[
A = \left\{
\begin{pmatrix} 1 & 0\\0 & 1 \end{pmatrix},\
\begin{pmatrix} 3 & 0\\0 & 3 \end{pmatrix},\
\begin{pmatrix} 3 & 2\\2 & 1 \end{pmatrix},\
\begin{pmatrix} 1 & 2\\2 & 3 \end{pmatrix}
\right\}
\]
and
\[
B = \left\{
\begin{pmatrix} 1 & 0\\0 & 1 \end{pmatrix},\
\begin{pmatrix} 1 & 2\\2 & 1 \end{pmatrix},\
\begin{pmatrix} 3 & 2\\0 & 3 \end{pmatrix},\
\begin{pmatrix} 3 & 0\\2 & 3 \end{pmatrix}
\right\}
\]
are each isomorphic to $C_2^2$, each normalized by $H$,
and commute with one another.
Looking at the action of $H$ on $B$, we conclude that
$HB$ is isomorphic to $S_4 \cong C_2^2 \rtimes S_3$.
The group $A$ is a normal complement to $HB$ and so we conclude that
$\GL_2(\bbZ/4\bbZ) \cong C_2^2 \rtimes S_4$.

We also see that there is a non-stem non-split central
extension of $S_4$ that is isomorphic to $\SL_2(\bbZ/4\bbZ)$.
We have an alternate description 
as $\SL_2(\bbZ/4\bbZ) \cong A_4 \rtimes C_4$
where $C_4$ is generated by a lift of a transposition in $S_4$.

\begin{lem} \label{lem:ext_poly_by_2}
Suppose $G$ sits in an exact sequence
\[
1 \to C_2 \to G \to P \to 1 \ .
\]
If $P \cong A_4$ then $G$ is isomorphic to
\[
A_4 \times C_2,\textrm{ or }
\widetilde{A_4}.
\]
If $P \cong S_4$ then $G$ is isomorphic to
\[
S_4 \times C_2,\
\widetilde{S_4}_+,\
\widetilde{S_4}_-,\textrm{ or }
\SL_2(\bbZ/4\bbZ)
\]
If $P \cong A_5$ then $G$ is isomorphic to
\[
A_5 \times C_2,\textrm{ or }
\widetilde{A_5}.
\]
\end{lem}

\begin{proof}
As in Lemma~4.4~of~\cite{DolIsk}, we have
\[
H^2(A_4,C_2) \cong C_2,\
H^2(S_4,C_2) \cong C_2^2,\textrm{ and }
H^2(A_5,C_2) \cong C_2.
\]
We have exhibited as many isomorphism classes as there are extensions,
so the result is proved.
\end{proof}

\begin{lem} \label{lem:ext_poly_by_22}
Suppose $G$ sits in an exact sequence
\[
1 \to N \to G \to P \to 1 ,
\]
where $N\cong C_2^2$. If $P \cong A_4$ then $G$ is isomorphic to
\[
A_4 \times C_2^2,\
\widetilde{A_4} \times C_2,\
C_2^2 \rtimes A_4,\textrm{ or }
C_4^2 \rtimes C_3.
\]
If $P \cong S_4$ then $G$ is isomorphic to
\begin{gather*}
S_4 \times C_2^2,\
\widetilde{S_4}_+ \times C_2,\
\SL_2(\bbZ/4\bbZ) \times C_2,\
\widetilde{S_4}_- \times C_2,\\
\widetilde{A_4} \rtimes C_4,\
\widetilde{S_4}_+ \rtimes C_2,\
\GL_2(\bbZ/4\bbZ),\
C_2^2 \rtimes S_4,\textrm{ or }
C_4^2 \rtimes S_3.
\end{gather*}
If $P \cong A_5$ then $G$ is isomorphic to
\[
A_5 \times C_2^2,\textrm{ or }
\widetilde{A_5} \times C_2.
\]
\end{lem}

\begin{proof}
Recall that $\Aut(C_2^2) \cong S_3$.

{\bf Case $P \cong A_4$:}\\
\noindent
First, assume that $A_4 \to \Aut(N)$ has trivial image.
In this case, $G$ can be described inductively as
a central extension by $C_2$ of a central extension by $C_2$.
Thus $G$ is an extension of $A_4 \times C_2$ or $\widetilde{A_4}$
by $C_2$.
In the first case, the subgroup $C_2$ has an isomorphic lift since $N
\cong C_2^2$.  Thus $G$ is isomorphic to $\widetilde{A_4} \times C_2$
or $A_4 \times C_2^2$.
In the latter case, we have $H^2(\widetilde{A_4},C_2)=1$
as in Lemma~4.4~of~\cite{DolIsk}.
Thus, $\widetilde{A_4}$ has no non-trivial central extensions. 

Now, suppose that $A_4 \to \Aut(N)$ is non-trivial;
then the image is isomorphic to $C_3$.
Let $w$ be the preimage in $G$ of an involution in $A_4$
and let $r$ be an element of order $3$ in $G$.
Recall that the involutions in $A_4$ generate the Klein four group.
Denote by $r(w)=rwr^{-1}$.
Since $r$ permutes the involutions in $A_4$,
we have $w r(w)=n r^2(w)$ and
$w^{-1} r(w)^{-1}=n' r^2(w)^{-1}$
for some $n,n' \in N$.
The involutions in $A_4$ commute with elements of $N$,
thus $[w,r(w)]=nn'$.  Thus
\[
w^2 r(w)^2 n n' = wr(w)wr(w)=(nr^2(w))^2=r^2(w)^2n^2=r^2(w)^2 \ .
\]
Recalling that $w^2=w^{-2}$, since $w^2 \in N$, we obtain
\[
w^2r(w^2)r^2(w^2)=nn'.
\]
The product of the $r$-orbit in $N$ is trivial so $nn'=1$.
We conclude that $w$ and $r(w)$ commute. 

The involutions in $A_4$ are all conjugate,
so the orders of their preimages in $G$ are all the same.
If the order is $2$ we obtain $G\cong C_2^2 \rtimes A_4$.
Otherwise the order is $4$ and we obtain $G\cong C_4^2 \rtimes C_3$.

{\bf Case $P \cong S_4$:}\\
\noindent
Let $\pi : G \to S_4$ be the natural projection,
let $H = \pi^{-1}(A_4)$, and let $g=\pi^{-1}(12)$
be a preimage of a transposition.
Now $G=\langle H,g \rangle$ and $H$ is isomorphic to one of
\[
A_4 \times C_2^2,\
\widetilde{A_4} \times C_2,\
C_2^2 \rtimes A_4,\
C_4^2 \rtimes C_3,
\]

{\bf Subcase $H \cong A_4 \times C_2^2$:}\\
\noindent
In fact, we will write $H = A_4 \times N$ where $A_4$ is an isomorphic
lift of $A_4$ from $S_4$.
Suppose $g$ acts trivially on $N$.
Let $K=\langle A_4,g\rangle$.
If $g^2=1$, then $K \simeq S_4$ and we obtain $G \cong S_4 \times N$.
If $g^2=n \ne 1$ then $K=\langle A_4,g\rangle$ is a non-split extension
of $S_4$ by $C_2$ containing $A_4$;
thus $K \cong \SL_2(\bbZ/4\bbZ)$.
Since $K$ has index $2$ in $G$, it is normal and
we conclude that $G \cong \SL_2(\bbZ/4\bbZ) \times C_2$.

Now suppose $g$ acts non-trivially on $N$.
Now $\langle N, g \rangle \cong D_8$.
By possibly choosing a different preimage for $g$, we may assume
$g^2=1$.
Thus there is only one possibility:
$G \cong C_2^2 \rtimes S_4 \cong \GL_2(\bbZ/4\bbZ)$.

{\bf Subcase $H \cong \widetilde{A_4} \times C_2$:}\\
\noindent
We will write $H = \widetilde{A_4} \times \langle w \rangle$
where $w$ is an involution in the center of $H$.
Let $z$ be the non-trivial involution in the center of
$\widetilde{A_4}$.
Assume $g$ acts trivially on $N$.
If $g^2=1$ then $G \cong \widetilde{S_4}_+ \times C_2$.
If $g^2=z$ then $G \cong \widetilde{S_4}_- \times C_2$.
If $g^2=w$ or $g^2=zw$ then
$G = \widetilde{A_4} \rtimes \langle g \rangle
\cong \widetilde{A_4} \rtimes C_4$.

Now suppose $g$ acts non-trivially on $N$.  Since $z$ is canonical,
we must have $g(z)=z$ and $g(w)=zw$.  By possibly replacing
$g$ with $gw$ we may assume $g^2=1$.
The subgroup $\langle \widetilde{A_4}, g \rangle$ is of index $2$ in $G$
and is isomorphic to $\widetilde{S_4}_+$.
Thus $G = \widetilde{S_4}_+ \rtimes \langle w \rangle
\cong \widetilde{S_4}_+ \rtimes C_2$.

{\bf Subcase $H \cong C_2^2 \rtimes A_4$:}\\
\noindent
We may write $H = N \rtimes A_4$ where $A_4$ is an isomorphic lift
from $S_4$.
Let $K$ be the normal subgroup of $A_4$ isomorphic to $C_2^2$.
Note that $K$ is also normal in $S_4$, thus $G/K$ is an extension
of $N$ by $S_3$.
The only possibility is $G/K\cong N \rtimes S_3 \cong S_4$.
By replacing $g$ with a different preimage we can assume
that $g^2K \notin NK$.  Since $g^2 \notin K$ we have $g^2=1$.
Thus $G = N \rtimes S_4 \cong C_2^2 \rtimes S_4$.

{\bf Subcase $H \cong C_4^2 \rtimes C_3$:}\\
\noindent
Here $G$ is an extension of $C_4^2$ by $S_3$ where
$C_4^2$ sits inside the $S_3$-equivariant exact sequence
\[
1 \to C_2^2 \to C_4^2 \to C_2^2 \to 1
\]
where $S_3$ acts faithfully on both instances of $C_2^2$.
Taking the long exact sequence in group cohomology we have
\[
\cdots \to H^2(S_3,C_2^2) \to H^2(S_3,C_4^2) \to H^2(S_3,C_2^2) \to \cdots
\]
where $H^2(S_3,C_2^2)=1$ since $S_4$ is the only extension of
$S_3$ by $C_2^2$ with the appropriate action.
We conclude that $H^2(S_3,C_4^2)$ is trivial and thus
$G \cong C_4^2 \rtimes S_3$ is the only possibility.

{\bf Case $P \cong A_5$:}\\
\noindent
Note that $A_5$ is simple, so there is no non-trivial homomorphism
$A_5 \to \Aut(C_2^2) \cong S_3$.
The argument is now identical to that of $A_4$ in this case.
\end{proof}

It remains to determine the representation dimensions of all these
groups.

\begin{lem} \label{lem:polyhedral_ext}
Suppose $G$ sits inside an exact sequence
\[
1 \to N \to G \to P \to 1
\]
where $P$ is a $A_4$, $S_4$ or $A_5$ and $N \cong C_2$
or $N \cong C_2^2$.
Suppose that $\bbk$ has characteristic $0$.
If $P$ is a subgroup of $\PGL_2(\bbk)$,
then $\rdim_{\bbk}(G) \le 6$.
\end{lem}

\begin{proof}
First, we note that since $P \subseteq \PGL_2(\bbk)$,
the corresponding binary polyhedral group $\widetilde{P}$
is a subgroup of $\GL_2(\bbk)$.
Thus we may assume
$\rdim_\bbk(\widetilde{P}) \le 2$ and $\rdim_\bbk(P) \le 3$.

Since $\rdim_\bbk(C_2)=1$, we have $\rdim_{\bbk}(G) \le 5$
for all $G$ among
\begin{gather*}
A_4 \times C_2,\ \widetilde{A_4},\
A_4 \times C_2 \times C_2,\ \widetilde{A_4} \times C_2,\\
S_4 \times C_2,\ \widetilde{S_4}_+,\
S_4 \times C_2 \times C_2,\ \widetilde{S_4}_+ \times C_2,\\
A_5 \times C_2,\ \widetilde{A_5},\
A_5 \times C_2 \times C_2,\ \widetilde{A_5} \times C_2 \ .
\end{gather*}

By using an induced representation, we see that
$\rdim_\bbk(\widetilde{S_4}_+ \rtimes C_2) \le
2\rdim_\bbk(\widetilde{S_4}_+)=4$.

Let $\pi : \widetilde{S_4}_- \to S_4$ be the canonical map.
Then $\pi^{-1}(A_4) \cong \widetilde{A_4}$ is a subgroup of
$\widetilde{S_4}_-$ of index $2$.  If $S_4 \subseteq \PGL_2(\bbk)$
then so must $A_4$.
Thus, by an induced representation
$\rdim_\bbk(\widetilde{S_4}_-) \le 2\rdim_\bbk(\widetilde{A_4})=4$.
Thus we have $\rdim_{\bbk}(G) \le 5$ for $G$ among
\[
\widetilde{S_4}_-,\ \widetilde{S_4}_- \times C_2
\]

Consider the group $G=\widetilde{A_4} \rtimes C_4$.
There is a surjective homomorphism $G \to H$
where $H \cong \widetilde{A_4} \rtimes C_2$
since $C_4$ does not act faithfully on $\widetilde{A_4}$.
Thus, by an induced representation
$\rdim_\bbk(H) \le 2\rdim_\bbk(\widetilde{A_4})=4$.
We now recall that $\rdim_\bbk(C_4) \le 2$ when $\bbk$ has
characteristic $0$.
Moreover, we have an injective homomorphism $G \to H \times C_4$
thus $\rdim_\bbk(G) \le 4+2=6$.

Suppose $G$ is one of
\[ C_2^2 \rtimes A_4, C_2^2 \rtimes S_4 . \]
In these cases, there is evident normal subgroup $N \cong C_2^2$.
Additionally, there is a normal subgroup $M \cong C_2^2$ of $A_4$ or
$S_4$ which acts trivially on $N$ and is therefore also normal in $G$.
This gives two homomorphisms $G\to S_4$ obtained by taking the
quotients by $N$ and $M$.
The product provides an injective homomorphism $G \to S_4 \times S_4$.
Thus $\rdim_\bbk(G) \le 2\rdim_\bbk(S_4)=6$.

Now we consider $C_4^2 \rtimes S_3$.
Among all the surjective morphisms $\sigma: C_4^2 \to C_4$,
we may select one that is invariant under an involution in $S_3$.
Thus, there is a set of three surjective morphisms
$\sigma_1,\sigma_2,\sigma_3 : C_4^2 \to C_4$
that are invariant under $S_3$.
Since $C_4$ has a two-dimensional representation over $\bbk$,
we may construct a $6$-dimensional representation of $C_4^2 \rtimes S_3$
via the direct sum $\sigma_1 \oplus \sigma_2 \oplus \sigma_3$
where $S_3$ acts by permutation matrices.
Thus for $G$ among
\[
C_4^2 \rtimes C_3,\ C_4^2 \rtimes S_3
\]
we have $\rdim_\bbk(G) \le 3\rdim_\bbk(C_4) \le 6$.

Recall that $\SL_2(\bbZ/4\bbZ) \cong A_4 \rtimes C_4$.
Thus there is an injective homomorphism
$\SL_2(\bbZ/4\bbZ) \to S_4 \times C_4$.
We conclude that $\rdim_\bbk(\SL_2(\bbZ/4\bbZ)) \le 5$.
Thus for $G$ among
\[
\SL_2(\bbZ/4\bbZ),\ \SL_2(\bbZ/4\bbZ) \times C_2
\]
we have $\rdim_\bbk(G) \le 6$.

Now consider $G \cong \GL_2(\bbZ/4\bbZ)$.
As discussed above,
there is a surjective homomorphism $\pi : G \to S_4$
with kernel
\[
A =
\left\langle
\begin{pmatrix}
1 & 2 \\ 2 & 3
\end{pmatrix},\
\begin{pmatrix}
3 & 2 \\ 2 & 1
\end{pmatrix}
\right\rangle \ .
\]
There is also a normal subgroup
\[
N :=
\left\langle
\begin{pmatrix}
3 & 3 \\ 1 & 0
\end{pmatrix},\
\begin{pmatrix}
3 & 2 \\ 0 & 3
\end{pmatrix}
\right\rangle \subseteq HB,
\]
which is isomorphic to $A_4$.
The quotient $Q:=G/N$ is a group of order $8$ not isomorphic to
$C_2^3$
(in fact, $Q \cong D_8$ but we do not need this).
Thus $\rdim_\bbk(Q) \le 2$.
Note that $N \cap K = 1$ so there is an injective group homomorphism
\[
\pi : G \to S_4 \times Q \ .
\]
We conclude that
$\rdim_\bbk(G) \le \rdim_\bbk(S_4) + \rdim_\bbk(Q) \le 5$.
Thus if $G$ is among the following:
\[
\SL_2(\bbZ/4\bbZ),\ \SL_2(\bbZ/4\bbZ) \times C_2,\ \GL_2(\bbZ/4\bbZ)
\]
then $\rdim_\bbk(G) \le 6$.
\end{proof}

\begin{table}[htb]
\caption{Extensions of polyhedral groups}
\label{tbl:poly_exts}
\def\arraystretch{1.5}
    \centering
    \begin{tabular}{|c|c|c|cc|cc|}
\hline
$N$ & $P$ & $G$ & Order & GAP ID & $\rdim_{\mathbb{Q}}(G)$ & $\rdim_{\bbk}(G)$\\
\hline \hline
$1$ & $A_4$ & $A_4$ & $12$ & $3$ & $3$ & $3$\\ \cline{2-7}
 & $S_4$ & $S_4$ & $24$ & $12$ & $3$ & $3$\\ \cline{2-7}
 & $A_5$ & $A_5$ & $60$ & $5$ & $4$ & $3$\\
\hline
$C_2$ & $A_4$ & $A_4 \times C_2$ & $24$ & $13$ & $3$ & $3$\\
 &  & $\widetilde{A_4}$ &  & $3$ & $4$ & $2$\\ \cline{2-7}
 & $S_4$ & $S_4 \times C_2$ & $48$ & $48$ & $3$ & $3$\\
 &  & $\widetilde{S_{4+}}$ &  & $29$ & $4$ & $4$\\
 &  & $\widetilde{S_{4-}}$ &  & $28$ & $8$ & $4$\\
 &  & $\SL_2(\mathbb{Z}/4\mathbb{Z})$ &  & $30$ & $5$ & $5$\\ \cline{2-7}
 & $A_5$ & $A_5 \times C_2$ & $120$ & $35$ & $4$ & $3$\\
 &  & $\widetilde{A_5}$ &  & $5$ & $4$ & $2$\\
\hline
$C_2^2$ & $A_4$ & $A_4 \times C_2^2$ & $48$ & $49$ & $4$ & $4$\\
 &  & $\widetilde{A_4} \times C_2$ &  & $32$ & $4$ & $3$\\
 &  & $C_2^2 \rtimes A_4$ &  & $50$ & $6$ & $6$\\
 &  & $C_4^2 \rtimes C_3$ &  & $3$ & $3$ & $6$\\ \cline{2-7}
 & $S_4$ & $S_4 \times C_2^2$ & $96$ & $226$ & $4$ & $4$\\
 & & $\widetilde{S_{4+}} \times C_2$ &  & $189$ & $5$ & $5$\\
 &  & $\widetilde{S_{4-}} \times C_2$ &  & $188$ & $9$ & $5$\\
 &  & $\SL_2(\mathbb{Z}/4\mathbb{Z}) \times C_2$ &  & $194$ & $5$ & $6$\\
 &  & $\widetilde{A_4} \rtimes C_4$ &  & $66$ & $6$ & $6$\\
 &  & $\widetilde{S_{4+}} \rtimes C_2$ &  & $190$ & $8$ & $4$\\
 &  & $\GL_2(\mathbb{Z}/4\mathbb{Z})$ &  & $195$ & $5$ & $5$\\
 &  & $C_2^2 \rtimes S_4$ &  & $227$ & $6$ & $6$\\
 &  & $C_4^2 \rtimes S_3$ &  & $64$ & $6$ & $6$\\ \cline{2-7}
 & $A_5$ & $A_5 \times C_2^2$ & $240$ & $190$ & $5$ & $4$\\
 &  & $\widetilde{A_5} \times C_2$ &  & $94$ & $5$ & $3$\\
\hline
   \end{tabular}

\end{table}


\end{document}